\documentclass[a4paper]{amsart}
\usepackage{amsmath,amssymb,amsthm,enumerate}
\title[GJMS operators and $Q$-curvature]
{{\bf GJMS operators and $Q$-curvature for conformal Codazzi structures}}
      
\author{TAIJI MARUGAME}
\date{}

\newcommand\R{\mathbb{R}}

\newcommand\Ric{{\rm Ric}}
\newcommand\scal{{\rm Scal}}
\renewcommand\a{\alpha}
\renewcommand\b{\beta}
\newcommand\g{\gamma}
\renewcommand\d{\delta}
\newcommand\e{\epsilon}
\renewcommand\r{\rho}
\renewcommand\th{\theta}
\newcommand\br{\boldsymbol{\rho}}
\newcommand\bu{\boldsymbol{u}}
\newcommand\bh{\boldsymbol{h}}

\newcommand\U{\Upsilon}
\newcommand\lam{\lambda}

\newcommand{\calE}{\mathcal{E}}
\newcommand{\calJ}{\mathcal{J}}
\newcommand{\calO}{\mathcal{O}}
\newcommand{\wt}{\widetilde}
\newcommand{\wh}{\widehat}

\newtheorem{thm}{Theorem}[section]

\newtheorem{prop}[thm]{Proposition}

\newtheorem{lem}[thm]{Lemma}
\newtheorem{cor}[thm]{Corollary}

\makeatletter
\@addtoreset{equation}{section}

\makeatother
\address{Graduate School of Mathematical Sciences, The University of Tokyo,
	3-8-1 Komaba, Meguro, Tokyo 153-8914, Japan}
\email{marugame@ms.u-tokyo.ac.jp}

\begin{document}
\maketitle 
\begin{abstract}
The conformal Codazzi structure is an intrinsic geometric structure on strictly convex hypersufaces in a locally flat projective manifold. We construct the GJMS operators and the $Q$-curvature for conformal Codazzi structures by using the ambient metric. We relate the total $Q$-curvature to the logarithmic coefficient in the volume expansion of the Blaschke metric, and derive the first and second variation formulas for a deformation of strictly convex domains.
\end{abstract}
\section{Introduction}
The ambient metric is a powerful tool for constructing geometric invariants and invariant differential operators such as GJMS operators (\cite{GJMS}, \cite{GG}) and $Q$-curvature (\cite{FH}) in conformal and CR geometries. It was first introduced by Fefferman \cite{Fe} for CR manifolds via the solution to the complex Monge--Amp\`ere equation, and then 
the Fefferman--Graham ambient metric was constructed for conformal manifolds by solving Ricci flat equation formally on the ambient space (\cite{FG1}, \cite{FG3}). 

In this article we define the ambient metric for another geometric structure: the {\it conformal 
Codazzi structure}. We will quickly review this geometric structure in dimension greater than 3; 
We refer the reader to \cite{BC} and \cite{M} for the details including low dimensional cases. 
Let $M$ be a $C^\infty$-manifold of dimension $n\ge4$. We use abstract index notation to denote tensors and tensor bundles. For example, we denote $TM$ and $S^3T^\ast M$ by $\calE^\a$ and $\calE_{(\a\b\g)}$ respectively. We also use these symbols to denote the space of sections of the bundles. The conformal density bundle of weight $w$ is defined by $\calE[w]=(\wedge^n T^\ast M)^{-w/n}$, and we put $[w]$ to a bundle to express the tensor product with $\calE[w]$. A conformal Codazzi structure on $M$ is a conformal structure $[h]$ together with a trace-free symmetric 3-tensor $A_{\a\b\g}\in\calE_{(\a\b\g)_0}[2]$, called the {\it Fubini--Pick form}, which satisfies the following {\it Gauss--Codazzi equations}:
\begin{align*}
&({\rm Gauss}) &
 2\,{\rm tf}(A_{\nu\g[\a}A_{\b]\mu}{}^\nu)+W^h_{\a\b\g\mu}&=0,  \\
&({\rm Codazzi}) &
\nabla^h_{[\a}A_{\b]\g\mu}-\frac{1}{n}\Bigl(h_{\mu[\a}(\d A)_{\b]\g}+h_{\g[\a}(\d A)_{\b]\mu}\Bigr)&=0,
\end{align*}
where $\bh_{\a\b}\in\calE_{\a\b}[2]$ is the conformal metric, $W^h_{\a\b\g\mu}$ is the Weyl tensor and $(\d A)_{\a\b}=\nabla^h_\g A_{\a\b}{}^\g$. The conformal Codazzi structure is an intrinsic geometric structure on a strictly convex hypersurface in a locally flat projective manifold; The projective second fundamental form induces a conformal structure and the Gauss--Codazzi equations come from the flatness of the ambient projective structure. Conversely, the conformal structure and the Fubini--Pick form recover the local immersion to the projective space (projective Bonnet theorem).

In this paper we consider the case where $M$ is globally realized as the boundary of a strictly convex domain $\Omega$ in a manifold $N$ with a locally flat projective structure $[\nabla]$. As in the CR case, the ambient metric for the conformal Codazzi structure is constructed from the solution to the Monge--Amp\`ere equation on $\Omega$. The real Monge--Amp\`ere equation for a projective density $\br\in\calE(2)$ is defined by 
\begin{equation}\label{real-MA}
\begin{aligned}
\calJ[\br]&:=\det(D_ID_J\br) =\det\begin{pmatrix} 2\br & \nabla_j\br \\
                              \nabla_i\br & \nabla_i\nabla_j\br+2\br P_{ij}
                                \end{pmatrix}=-1, \\
\Omega&=\{\br<0\},
\end{aligned}
\end{equation}
where $D_I$ is the projective $D$-operator and $P_{ij}=(1/n)\Ric_{ij}$ is the projective Schouten 
tensor of $\nabla\in[\nabla]$. As in the complex case (\cite{Fe}), one can construct an approximate solution to \eqref{real-MA}:
\begin{prop}[\cite{M}]
There exists a defining density $\br\in\calE(2)$ such that
\begin{equation}\label{approxMA}
\calJ[\br]=\begin{cases}
-1+O(\br^{\infty}) & {\text if }\ n\ {\text is\ odd}, \\
-1+\calO\br^{n/2+1} &  {\text if }\ n\ {\text is\ even},
\end{cases}
\end{equation}
where $\calO\in\calE(-n-2)$. Moreover, such a density is unique modulo $O(\br^{\infty})$ for $n$ odd, and unique modulo $O(\br^{n/2+2})$ for $n$ even.
\end{prop}
We call $\br$ a {\it Fefferman defining density} when it satisfies \eqref{approxMA}. The density $\mathcal{O}$ is called the {\it obstruction density} and the boundary value 
$\mathcal{O}|_M\in\calE[-n-1]$ gives a local conformal invariant.

We define the ambient metric for a conformal Codazzi structure by the projective tractor 
$$
\wt g_{IJ}=D_ID_J\br
$$ 
with a Fefferman defining density $\br$. It is identified with a homogeneous Lorentz metric on the projective density bundle and allows us to construct the GJMS operators and the $Q$-curvature as in conformal and CR cases. The GJMS operator $P_m:\calE[m-n/2]\rightarrow\calE[-m-n/2]$ is a conformally invariant linear differential operator whose symbol agrees with that of the power of the Laplacian $\Delta_h^m$, and is different from usual conformal GJMS operator. In particular, we can define $P_{n/2+2}$ for even $n$, which does not exist for general conformal manifolds (\cite{GH}). The conformal Codazzi $Q$-curvature integrates to a global conformal invariant and is also different from usual conformal $Q$-curvature.

Recall that conformal and CR $Q$-curvatures are intimately related to the geometry of complete metrics inside the domains, such as Poincar\'e-Einstein metric (\cite{G2}, \cite{GZ}, \cite{FG2}) and the Cheng--Yau metric (\cite{CY}, \cite{HPT}).  In our case, the relevant metric is the {\it Blaschke metric}, which is a projectively invariant metric on $\Omega$ defined by 
$$
g_{ij}=\frac{\nabla_i\nabla_j\br}{-2\br}+\frac{\nabla_i\br\nabla_j\br}{4\br^2}-P_{ij}.
$$
In \cite{M}, the author shows that the Blaschke metric has the volume expansion 
\begin{equation*}
{\rm Vol}(\{\tau^{-2}\br<-{\textstyle\frac{1}{2}}\e^2\}) 
=\sum_{j=0}^{\lceil n/2\rceil-1}c_{-n+2j}\e^{-n+2j}+
\begin{cases}V+o(1) & (n: {\rm odd}) \\
L\log(1/\e)+V+o(1) & (n: {\rm even})
\end{cases}
\end{equation*}
for a fixed projective scale $\tau\in\calE(1)$, and $V$ is independent of the choice of $\tau$ when $n$ is odd while $L$ is independent of the choice of $\tau$ when $n$ is even. 
We prove that the integral $\overline{Q}$ of the conformal Codazzi $Q$-curvature agrees with  
a multiple of $L$. By using this fact, we show that the first variation of $\overline{Q}$ for a deformation of a strictly convex domain is given by the obstruction density $\calO$ (Theorem \ref{first-var}). As a corollary, it is shown that a strictly convex surface in $\R^3$ is a critical point of $\overline{Q}$ if and only if it is projectively equivalent to the sphere (Corollary \ref{cor-first-var}). We also show that the second variation for a deformation parametrized by a density on the boundary is described by the GJMS operator $P_{n/2+2}$ (Theorem \ref{L-second-var}). For the spheres, the operator agrees with the ordinary GJMS operator, and it turns out that its kernel corresponds to a deformation via a family of projective linear transformations, which gives a trivial deformation of the conformal Codazzi manifold.
\medskip

This paper is organized as follows. In \S2, we review the projective tractor calculus and geometry of  
 strictly convex domains, including some tensor identities given in \cite{M}. In \S3, we define the GJMS operators and the $Q$-curvature using the ambient metric in conformal Codazzi geometry. 
Then we prove the self-adjointness of the GJMS operators and show that the total $Q$-curvature agrees with a multiple of the coefficient $L$ in the volume expansion of the Blaschke metric. 
Finally in \S4, we derive first and second variation formulas of the total $Q$-curvature under a deformation of strictly convex domains in a locally flat projective manifold.
\medskip

\noindent {\it Notations.} We adopt Einstein's summation convention and assume that 
\begin{itemize}
\item
uppercase Latin indices $I, J, K, \dots$ run from 0 to $n+1$;
\item
lowercase Latin indices
$i, j, k, \dots$
run from 1 to $n+1$; 
\item
lowercase Greek indices $\alpha,\beta,\gamma,\dots$ run from 1 to $n$. 
\end{itemize}
\medskip
\noindent{\bf Acknowledgments.} This paper is based on part of the author's thesis at the University of Tokyo. I would like to express my deep gratitude to Professor Kengo Hirachi for his continuous encouragement and helpful advice. I would also like to thank Dr. Yoshihiko Matsumoto, Professor Michael Eastwood, and Professor Bent \O rsted  for invaluable comments on the results. This research was partially supported by JSPS Fellowship and
 KAKENHI 13J06630.

\section{Preliminaries}
\subsection{Projective tractor calculus}
We quickly review the tractor calculus in projective differential geometry; We refer to \cite{BEG} for the detail of the constructions. A projective structure $[\nabla]$ on an oriented differentiable manifold $N$ is a projective equivalence class of torsion free affine connections, where two affine connections are said to be projectively equivalent when they have the same geodesic paths. It is known that $\nabla$ and $\nabla^\prime$ are projectively equivalent if and only if there exists a 1-form $p$ such that 
$$
\nabla^\prime_i X^j=\nabla_i X^j+p_iX^j+p_lX^l \d_i{}^j,
$$
for any $X^i\in\calE^i$, which we write as $\nabla^\prime=\nabla+p$. We define the 
{\it projective density bundle} of weight $w$ as an oriented real line bundle $\calE(w)=(\wedge^{n+1} T^\ast N)^{-(w/n+2)}$, where $n+1=\dim N$. A choice of a positive section $\tau\in\calE(1)$, called a {\it projective scale}, determines a unique representative connection $\nabla\in[\nabla]$ satisfying $\nabla\tau=0$. We deal with only such connections in $[\nabla]$. If we change the scale as $\wh\tau=e^{-\U}\tau$ with $\U\in C^\infty(N)$, the connection transforms as $\wh\nabla=\nabla+d\U$.
A projective structure is {\it locally flat} if the projective Weyl curvature 
$$
C_{ij}{}^k{}_l=R_{ij}{}^k{}_l-2\,\d_{[i}{}^kP_{j]l}
$$
vanishes, where $P_{ij}=(1/n)\Ric_{ij}$ is the {\it projective Schouten tensor}. This is equivalent to the condition that around each point in $N$ there exists a local {\it affine scale}, namely a projective scale such that the corresponding representative connection is flat. 

We define the {\it ambient space} $\wt N$ by the $\R_+$-bundle $\calE(-1)_+$ over $N$.
Then the {\it projective tractor bundle} is the rank $n+2$ vector bundle $\wt\calE^I=T\wt N/\R_+$ over $N$, where $s\in\R_+$ acts on $T\wt N$ by $s^{-1}(\d_s)_\ast$ with the dilation $\d_s$ on $\wt N$. A section $\nu^I\in\wt\calE^I(w):=\wt\calE^I\otimes\calE(w)$, called a {\it tractor}, can be identified with a vector filed on $\wt N$ homogeneous of degree $w-1$. If we take a projective scale $\tau$ and trivialize $\wt N$ by $\tau^{-1}$, the tractor bundle is decomposed as 
$$
\wt \calE^I\overset{\tau}{\cong}
\begin{matrix}
\calE^i(-1) \\
\oplus \\
\calE(-1),
\end{matrix}
$$
where $\calE^i$ denotes $TN$. Under a rescaling $\wh\tau=e^{-\U}\tau$, we have the transformation formula
\begin{equation*}
\begin{pmatrix}
\wh\nu^i \\
\wh\lam
\end{pmatrix}
=
\begin{pmatrix}
\nu^i \\
\lam-\U_k\nu^k
\end{pmatrix},
\end{equation*}
where $\U_k=d\U$. The tractor $T^I\in\wt\calE^I(1)$ expressed by ${}^t(0,1)$ does not depend on the choice of a projective scale and is called the {\it Euler field}. The {\it projective tractor connection} 
is a projectively invariant connection on $\wt\calE^I$ defined by
\begin{equation}\label{tract-conn}
\nabla_i
\begin{pmatrix}
\nu^j \\
\lam
\end{pmatrix}
=
\begin{pmatrix}
\nabla_i\nu^j+\lam\d_i{}^j \\
\nabla_i\lam-P_{ik}\nu^k
\end{pmatrix}.
\end{equation}
We also have the induced tractor connections on the dual bundle $\wt\calE_I$ or various tensor bundles such as $\wt\calE_I{}^J$, computed from \eqref{tract-conn}. The tractor connection is flat if and only if the projective structure $[\nabla]$ is locally flat.

There is another important projectively invariant differential operator: The {\it projective 
$D$-operator} $D_I : \wt\calE_\ast(w)\rightarrow\wt\calE_{I\ast}(w-1)$ is defined by 
$$
D_If_\ast=
\begin{pmatrix}
wf_\ast \\
\nabla_i f_\ast
\end{pmatrix}
\in\, \begin{matrix}
\wt\calE_\ast(w) \\
\oplus \\
 \wt\calE_{i\ast}(w)
\end{matrix}
\overset{\tau}{\cong}
 \wt\calE_{I\ast}(w-1),
$$
where `$\ast$' denotes arbitrary tractor indices and $\nabla$ in the second slot is the connection 
induced from the tractor connection and the connection on $\calE(w)$ associated with $\tau$. 
We have the following useful formulas:
$$
T^ID_If_\ast=w f_\ast, \quad D_IT^J=\d_I{}^J.
$$
The $D$-operator can be regarded as a linear connection on the ambient space $\wt N$ whose flatness is equivalent to the local flatness of $[\nabla]$.

\subsection{Geometry of strictly convex domains}
Let $(N, [\nabla])$ be an oriented locally flat projective manifold of dimension $n+1$. We consider a relatively compact domain $\Omega$ in $N$ with smooth boundary $M$. We say $M$ is {\it strictly convex} if $\nabla_i\nabla_j \r|_{TM}$ is positive definite for a connection $\nabla\in[\nabla]$ and 
a defining function $\r$ which is negative in $\Omega$. The strict convexity does not depend on the choice of $\nabla$ and $\r$. If we fix a projective scale $\tau\in\calE(1)$, there is a canonical choice of a transverse vector field $\xi\in\Gamma(M, TN)$ with $d\r(\xi)>0$, called the 
{\it affine normal field}, which satisfies
\begin{align*}
\nabla_X Y&=\nabla^\xi_X Y-h(X, Y)\xi, \\
\nabla_X \xi&=S(X),        
\end{align*}
for $X, Y\in\Gamma(TM)$, where $\nabla^\xi_X Y, \,S(X)\in\Gamma(TM)$ and $h$ is a Riemannian 
metric on $M$ such that $\bigl(\xi\lrcorner\,\tau^{-(n+2)}\bigr)|_{TM}=vol_h$ and is called the {\it affine metric}. The connection $\nabla^{\xi}$ is called the {\it induced connection} and we define the {\it Fubini--Pick form} by the difference tensor $A=\nabla^h-\nabla^\xi$, where $\nabla^h$ is the Levi--Civita connection of $h$.  The endomorphism $S$ is called the {\it affine shape operator}. For a rescaling $\wh\tau=e^{-\U}\tau$, we have 
\begin{align*}
\wh h&=e^{2\U}h, \quad \wh A=A, \\
\wh S(X)&=e^{-2\U}\bigl(S(X)+(\xi\U-|d\U|_h^2)X+d\U(X)\,{\rm grad}_h\U-\nabla^\xi_X {\rm grad}_h\U\bigr),
\end{align*}
where ${\rm grad}_h\U$ denotes the gradient vector field of $\U|_M$ with respect to the metric $h$. Thus the boundary $M$ is endowed with the conformal structure $[h]$. The {\it conformal density bundle} is a real line bundle $\calE[1]=(\wedge^n T^\ast M)^{-1/n}$ and we have a canonical isomorphism $\calE(1)|_M\rightarrow\calE[1]$ by 
$\tau|_M\mapsto (vol_h)^{-1/n}$. Thus we identify $\calE(w)|_M$ with $\calE[w]$ for each $w$.
The weighted tensor $\bh_{\a\b}=\tau^2|_M h_{\a\b}\in\calE_{(\a\b)}[2]$ is conformally invariant and called the {\it conformal metric}. We raise and lower the indices with $\bh_{\a\b}$ and its inverse $\bh^{\a\b}$. By the flatness of $[\nabla]$, the affine shape operator $S_{\a\b}$ is symmetric and the Fubini--Pick form $A_{\a\b\g}$ becomes trace-free and totally symmetric.  
It is known that $A_{\a\b\g}$ vanishes if and only if $M$ is locally projectively equivalent to the sphere; see \cite{NS}.

Take a Fefferman defining density $\br=\tau^2\r\in\calE(2)$ of $\Omega$. We introduce a special local frame for $TN$ as follows. First we take a vector field $\xi$ on a neighborhood of $M$ satisfying  
\begin{equation}\label{xi}
\xi \r=1, \qquad \xi^i X^j\nabla_i\nabla_j\r=0 \quad {\rm if}\quad X\r=0.
\end{equation}
We set $r:=\xi^i\xi^j\nabla_i\nabla_j\r$ and call $r$ the {\it transverse curvature}. Since $\br$ satisfies \eqref{real-MA}, $\xi|_M$ coincides with the affine normal field (\cite[Lemma 4.5]{M}). Then 
we extend a local frame $\{e_\a\}$ for $TM$ by the parallel transport along the integral curves of $\xi$ with respect to $\nabla$. We call a local frame $\{e_\infty=\xi, e_\a\}$ thus obtained an {\it adapted frame}. In such a frame, the connection forms of $\nabla$ are given by
\begin{equation}\label{conn-adapted}
\begin{aligned}
\omega_\a{}^\b&=\omega^\xi{}_\a{}^\b, & \omega_\infty{}^\a&=S_\g{}^\a\th^\g+C^\a d\r, \\ 
\omega_\b{}^\infty&=-h_{\b\g}\th^\g, & \omega_\infty{}^\infty&=-r d\r,
\end{aligned}
\end{equation}
where $C^\a$ is a function, and $\omega^\xi{}_\a{}^\b$ restricts to the connection form of the induced connection on each level set and satisfies $\omega^\xi{}_\a{}^\b(\xi)=0$. Moreover, the ambient metric $\wt g_{IJ}=D_ID_J\br$ satisfies 
\begin{equation}\label{g-M}
\wt g_{IJ}|_M
=\begin{pmatrix}
0 & 1 & \quad \\
1 & \boldsymbol{r}|_M & \quad \\
\quad & \quad & \bh_{\a\b}
\end{pmatrix},
\quad 
\wt g^{IJ}|_M
=\begin{pmatrix}
-\boldsymbol{r}|_M  & 1 & \quad \\
1 & 0 & \quad \\
\quad & \quad & \bh^{\a\b}
\end{pmatrix},
\end{equation}
where $\boldsymbol{r}:=\tau^{-2}r$. We recall from \cite{M} that in an adapted frame, the components of the projective Schouten tensor and the affine shape operator satisfy 
\begin{align}
(n-1)P_{\a\b}|_M-{\rm Ric}^h_{\a\b}+(\d A)_{\a\b}+A_{\a\mu\nu}A_{\b}{}^{\mu\nu}-S_{\a\b}+({\rm tr}S)h_{\a\b}&=0, \label{Pab-M} \\
\xi\, S_\a{}^\b+S_\a{}^\g S_\g{}^\b+r S_\a{}^\b-\nabla^\xi_\a C^\b+\d_\a{}^\b P_{\infty\infty}&=0,
\label{xi-S} 
\end{align}
where $\Ric^h$ is the Ricci tensor of $h$ and $(\d A)_{\a\b}=\nabla^h_\g A_{\a\b}{}^\g$. Also, the transverse curvature satisfies the following equation:
\begin{multline}\label{r}
(n+2)r+\xi\log
\bigl(1-2r\r-4\r^2P_{\infty\infty}+8\r^3\,\wt h^{\g\mu}P_{\g\infty}P_{\mu\infty}\bigr) \\
+\xi\log\det\bigl(\d_\a{}^\b+2\r h^{\b\g}P_{\g\a}\bigr)-{\rm tr}S-h^{\a\b}P_{\a\b} \\
=-\Bigl(\frac{n}{2}+1\Bigr)\underline{\calO}\,\r^{n/2}+O(\r^{n/2+1}),
\end{multline}
where $\wt h^{\a\b}$ is the inverse of $\wt h_{\a\b}=h_{\a\b}+2\r P_{\a\b}$ and 
$\underline{\calO}=\tau^{n+2}\calO$. Setting $\r=0$ in this equation gives
\begin{equation}\label{r-M}
r|_M=\frac{2}{n}{\rm tr}S-\frac{1}{n(n-1)}\bigl(\scal_h-|A|^2\bigr),
\end{equation} 
where $\scal_h$ is the scalar curvature of $h$. 

Let $\wt\Delta=-\wt g^{IJ}\wt\nabla_I\wt\nabla_J$ be the Laplacian of $\wt g_{IJ}$. A projective scale $\tau$ is called a {\it harmonic scale} when it satisfies $\wt\Delta\tau=O(\br)$. In such a scale, the identity 
\begin{equation}\label{r-trS}
r|_M=\frac{1}{n}{\rm tr}S=\frac{1}{n(n-1)}(\scal_h-|A|^2)
\end{equation}
holds (\cite[Proposition 5.10]{M}).

\section{GJMS operators and $Q$-curvature}
\subsection{GJMS construction}
Let $M$ be the boundary of a strictly convex domain $\Omega$ in an $(n+1)$-dimensional  locally flat projective manifold $N$. Let $\wt\nabla_I$ be the Levi--Civita connection of the ambient metric $\wt g_{IJ}=D_ID_J\br$, where $\br$ is a Fefferman defining density of $\Omega$. Since $D_I$ is flat, we have 
$\wt\nabla_I=D_I+\wt\Gamma_{IJ}{}^K$ with
$$
\wt\Gamma_{IJK}=\frac{1}{2}D_ID_JD_K\br\in\wt\calE_{(IJK)}(-1).
$$
When $n$ is even the Monge--Amp\`{e}re equation $\det \wt g=-1+\mathcal{O}\br^{n/2+1}$ implies 
\begin{equation}\label{trGamma}
\wt\Gamma_{IK}{}^K=-\frac{1}{2}\Bigl(\frac{n}{2}+1\Bigr)\mathcal{O}\br^{n/2}D_I\br
-\frac{1}{2}\br^{n/2+1}D_I\mathcal{O}+O(\br^{n/2+2}),
\end{equation}
while the right-hand side is replaced by $O(\br^\infty)$ when $n$ is odd. 
Since $T^I\wt\Gamma_{IJK}=0$, we have $\wt\Delta\br=-(n+2)$. 

As in conformal and CR cases, $\wt\Delta$ satisfies the following commutation relations:
\begin{lem}
For a density $\wt f\in\calE(w)$ and a positive integer $m$, we have
\begin{equation}\label{Lap-comm}
\begin{aligned}[ ]
[\wt\Delta, \br^m]\wt f&=-m(2w+2m+n)\br^{m-1}\wt f, \\
[\wt\Delta^m, \br]\wt f&=-m(2w-2m+n+4)\wt\Delta^{m-1}\wt f.
\end{aligned}
\end{equation}
\end{lem}
By using the above lemma, we can construct conformally invariant differential operators as follows: 
Let $f\in\calE[w]$ be a conformal density on $M$. We extend $f$ to a density $\wt f\in\calE(w)$ and apply $\wt\Delta^{m}$. 
By \eqref{Lap-comm}, we have
$$
\wt\Delta^m (\wt f+\phi\br)=\wt\Delta^m\wt f-m(2w-2m+n)\wt\Delta^{m-1}\phi+\br\wt\Delta^m\phi
$$
for any density $\phi\in\calE(w-2)$. Therefore $\wt\Delta^m\wt f|_M$ is independent of the choice of an 
extension when $w=m-n/2$, and defines a differential operator
$$
P_m:\calE[m-n/2]\longrightarrow\calE[-m-n/2],
$$
which we call the {\it (conformal Codazzi) GJMS operator}. 

The following proposition is proved in the same way as in 
\cite{GJMS} and \cite{FH}:
\begin{prop}\label{harm-ext}
Let $f\in\calE[m-n/2]$.

{\rm (i)} There exists an extension $\wt f\in\calE(m-n/2)$ of $f$ such that 
$$
\wt\Delta\wt f=O(\br^{m-1}).
$$
Moreover, such an extension is unique modulo $O(\br^m)$ and satisfies $c_m\br^{1-m}\wt\Delta\wt f|_M= P_m f$ with $c_m=2^{m-1}\{(m-1)!\}^2$. 

{\rm (ii)} Let $\tau\in\calE(1)$ be a projective scale and set $\r=\tau^{-2}\br$. Then there exist $A\in\calE(m-n/2)$ and $B\in\calE(-m-n/2)$ such that $A|_M=f$ and 
$$
\wt\Delta(A+B\br^m\log|\r|)=O(\br^\infty).
$$
Moreover, $A$ and $B$ are unique modulo $O(\br^m)$ and $O(\br^\infty)$ respectively, and it holds that 
$c_m^\prime B|_M=P_m f$ with $c_m^\prime=-2^m m\{(m-1)!\}^2$.
\end{prop}

Next  we examine the dependence on the choice of a Fefferman defining density. When $n$ is odd, the Fefferman defining density  is unique to infinite order, so $P_m$ does not depend on 
the choice of $\br$. On the other hand, we must restrict the range of $m$ when $n$ is even: 

\begin{prop}
Let $n$ be even. Then the GJMS operator $P_m$ is independent of the choice of a Fefferman defining
density $\br$ if $1\le m\le n/2$. 
\end{prop}
\begin{proof}
Let $\br^\prime$ be another Fefferman defining density and $\wt g_{IJ}^\prime$ the associated ambient 
metric. Since $\br^\prime=\br+O(\br^{n/2+2})$, the Laplacian of $\wt g_{IJ}^\prime$ satisfies 
$\wt\Delta^\prime=-{\wt g}^{\prime IJ} D_ID_J+O(\br^{n/2})=\wt\Delta+O(\br^{n/2})$. Let $f\in\calE[m-n/2]$. We extend $f$ to $\wt f\in\calE(m-n/2)$ so that 
$$
\wt\Delta\wt f=\psi \br^{m-1}, \quad \psi|_M=P_m f
$$
holds. Then $\wt f$ satisfies $\wt\Delta^\prime \wt f =(\psi+O(\br))\br^{m-1}$ if $m\le n/2$, so we have 
$P_m^\prime=P_m$.
\end{proof}

By following the CR case (\cite{HMM}), we will introduce a further normalization on $\br$ so that $P_{n/2+2}$ is well-defined for $n$ even.

\begin{lem}\label{strict}
Let $n$ be even. Then there exists a defining density $\br$ of $\Omega$ such that 
$$
\calJ[\br]=-1+{\mathcal{O}}\br^{n/2+1}, \quad \wt\Delta{\mathcal{O}}=O(\br),
$$
where $\wt\Delta$ is the Laplacian of $\wt g_{IJ}=D_ID_J\br$. Moreover, a defining density $\br^\prime$ satisfies the same equations if and only if there exists $\phi\in\calE(-n-2)$ such that 
$\br^\prime=\br+\phi\br^{n/2+2}$ and $\wt\Delta\phi=O(\br)$.
\end{lem}
\begin{proof}
Take a defining density $\br$ such that $\calJ[\br]=-1+\mathcal{O}
\br^{n/2+1}$. Then any Fefferman defining density is written in the form $\br^\prime=\br+\phi\br^{n/2+2}$ 
with a density $\phi\in\calE(-n-2)$. By using the formula in \cite[Lemma 3.2]{M}, we have 
$$
\calJ[\br^\prime]=-1+(\mathcal{O}+\br\wt\Delta\phi)\br^{n/2+1}+O(\br^{n/2+3})
$$
so the obstruction density of $\br^\prime$ satisfies $\mathcal{O}^\prime=\mathcal{O}+\br\wt\Delta\phi+O(\br^2)$. Then the commutation relation \eqref{Lap-comm} gives
\begin{equation}\label{Lap-O}
\wt\Delta^\prime\mathcal{O}^\prime=\wt\Delta\mathcal{O}+(n+6)\wt\Delta\phi+O(\br).
\end{equation}
Thus, setting $\phi=-(n+6)^{-2}\br\wt\Delta\mathcal{O}$, we obtain $\br^\prime$ such that $\wt\Delta^\prime\mathcal{O}^\prime=O(\br)$. The second statement of the lemma also follows from \eqref{Lap-O}.
\end{proof}

A defining density given as in the above lemma is called a {\it strict Fefferman defining density}. 

\begin{prop}
When $n$ is even, the operators $P_{n/2+2}$ is independent of the choice of a strict Fefferman defining density.
\end{prop}
\begin{proof}
Let $\br$ and $\br^\prime$ be strict Fefferman defining densities. By Lemma \ref{strict}, we can write as
$\br^\prime=\br+\phi\br^{n/2+2}$ with $\phi\in\calE(-n-2)$ satisfying $\wt\Delta\phi=O(\br)$. 
The corresponding obstruction densities satisfy $\mathcal{O}|_M=\mathcal{O}^\prime|_M$ and 
$\wt\Delta(\mathcal{O}-\mathcal{O}^\prime)=O(\br)$, so it holds that $\mathcal{O}-\mathcal{O}^\prime=O(\br^2)$. Thus, from \eqref{trGamma} we have
\begin{equation}\label{trGamma-diff}
\wt g^{\prime IJ}\wt\Gamma^\prime_{IJ}{}^K-\wt g^{IJ}\wt\Gamma_{IJ}{}^K=O(\br^{n/2+2}).
\end{equation}
The associated ambient metrics satisfy  
\begin{equation}\label{ambient-diff}
\begin{aligned}
\wt g^{\prime IJ}&=\wt g^{IJ}-\Bigl(\frac{n}{2}+2\Bigr)\Bigl(\frac{n}{2}+1\Bigr)\phi T^IT^J\br^{n/2} \\
&\quad-\Bigl(\frac{n}{2}+2\Bigr)(T^ID^J\phi + T^JD^I\phi+\phi\wt g^{IJ})\br^{n/2+1} 
+O(\br^{n/2+2}).
\end{aligned}
\end{equation}
Let $f\in\calE[2]$. By Proposition \ref{harm-ext} (i), there is an extension $\wt f$ of $f$ such that 
$\wt\Delta\wt f=\psi\br^{n/2+1}$ and $c_{n/2+2}\psi|_M=P_{n/2+2}f$. Then by \eqref{trGamma-diff} 
and \eqref{ambient-diff}, we have
\begin{align*}
\wt\Delta^\prime \wt f
&= \wt\Delta\wt f+\Bigl(\frac{n}{2}+2\Bigr)\Bigl(\frac{n}{2}+1\Bigr)
\phi T^IT^J D_ID_J\wt f\br^{n/2} \\
&\quad +\Bigl(\frac{n}{2}+2\Bigr)(2T^ID^J\phi D_ID_J\wt f 
-\phi\wt\Delta\wt f)\br^{n/2+1}+O(\br^{n/2+2}) \\
&=\psi\br^{n/2+1}+\Bigl(\frac{n}{2}+2\Bigr)(n+2)\phi\wt f\br^{n/2}+(n+4)D^I\phi D_I\wt f\br^{n/2+1} \\
&\quad+O(\br^{n/2+2}).
\end{align*}
Thus, setting $\wt f^\prime=\wt f+(n/2+2)\phi\wt f\br^{n/2+1}$, we obtain
\begin{align*}
\wt\Delta^\prime \wt f^\prime&=\wt\Delta^\prime \wt f-\Bigl(\frac{n}{2}+2\Bigr)(n+2)\phi\wt f\br^{n/2}
+\Bigl(\frac{n}{2}+2\Bigr)\wt\Delta(\phi\wt f)\br^{n/2+1} \\
&=\psi{\br^{\prime}}^{n/2+1}+O({\br^{\prime}}^{n/2+2}),
\end{align*}
which implies $P^{\prime}_{n/2+2}f=P_{n/2+2}f$.
\end{proof}

It is proved in \cite{G1} and \cite{GH} that on a general even dimensional conformal manifold there is no conformally invariant linear differential operator whose principal part agrees with that of $\Delta_h^{m}$ for $m\ge n/2+1$. In our case, an additional structure on $M$, namely the Fubini--Pick form with the 
Gauss--Codazzi equations, enables us to define an over critical GJMS operator.
\ \\

As an example, we will compute a formula for $P_1:\calE[1-n/2]\longrightarrow\calE[-1-n/2]$. We fix a projective scale $\tau\in\calE(1)$, and take an adapted frame $\{e_\infty=\xi, e_\a\}$. 
For an arbitrary extension $\wt f\in\calE(1-n/2)$ of $f\in\calE[1-n/2]$, we have
$$
D_ID_J\wt f=
\begin{pmatrix}
n/2(n/2-1)\wt f & -n/2\nabla_j\wt f \\
-n/2\nabla_i\wt f & \nabla_i\nabla_j\wt  f+(1-n/2)P_{ij}\wt f
\end{pmatrix}.
$$
Thus,  using \eqref{conn-adapted}, \eqref{g-M}, \eqref{Pab-M} and \eqref{r-M}, we obtain
\begin{align*}
P_1 f&=-\wt g^{IJ}D_ID_J \wt f|_M \\
&=-\bh^{\a\b}\Bigl(\nabla_\a\nabla_\b\wt f+\Bigl(1-\frac{n}{2}\Bigr)P_{\a\b}\wt f\,\Bigr)+n\nabla_\infty\wt f
+\frac{n}{2}\Bigl(\frac{n}{2}-1\Bigr) \boldsymbol{r}\wt f \\
&=\Delta_h f+\frac{n-2}{4(n-1)}(\scal_h-|A|^2)f.
\end{align*}

\subsection{Self-adjointness}
We will prove that the GJMS operators are self-adjoint. 
First we relate the ambient Laplacian $\wt\Delta$ to the Laplacian $\Delta_g$ of the Blaschke metric.
\begin{lem}\label{Lap-relate-lem}
{\rm (i)} Let $u\in C^\infty(\Omega)$ and $w\in\R$. We assume $du=O(\br^{-1})$ in local coordinates around each boundary point. If $w\neq 0$, we further assume that $u=O(1)$. Then when $n$ is even we have
\begin{equation}\label{Lap-relate}
2(-\br)^{1-w/2}\wt\Delta\bigl((-\br)^{w/2}u\bigr)
=\bigl(\Delta_g+w(n+w)\bigr)u+O(\br^{n/2+1}).
\end{equation}
When $n$ is odd, the error term is replaced by $O(\br^\infty)$. \\
{\rm (ii)} Let $\wt f\in\calE(2)$. Then when $n$ is even we have
\begin{equation}\label{Lap-relate2}
2\wt\Delta\wt f=(\Delta_g+2(n+2))\bigl((-\br)^{-1}\wt f\,\bigr)+O(\br^{n/2+2}).
\end{equation}
\end{lem} 

\begin{proof}
(i) We consider the case of even $n$. We set $\tau^\prime:=(-\br)^{1/2}$. Then $\tau^\prime$ defines a projective scale inside $\Omega$.
We denote by $\nabla^\prime$ the associated connection, which is singular along the boundary. 
Since $\nabla^\prime\br=0$, the Blaschke metric is given by $g_{ij}=-P^\prime_{ij}$, where $P^\prime_{ij}$ is the projective Schouten tensor of $\nabla^\prime$. We set $\bu=(-\br)^{w/2}u$.
Then in the scale $\tau^\prime$ we have
$$
\wt g_{IJ}=\begin{pmatrix}
2\br & 0 \\
0 & -2\br g_{ij}
\end{pmatrix},
\quad D_ID_J\bu=
\begin{pmatrix}
w(w-1)\bu & (w-1)\nabla_j^\prime\bu\\
(w-1)\nabla_i^\prime\bu & \nabla_i^\prime\nabla_j^\prime\bu-wg_{ij}\bu
\end{pmatrix}.
$$
Thus, using \eqref{trGamma} and the assumption on $u$, we compute as
$$
2(-\br)^{1-w/2}\wt\Delta\bu
=-g^{ij}\nabla_i^\prime\nabla_j^\prime u+w(n+w)u+O(\br^{n/2+1}). 
$$
We shall compare $-g^{ij}\nabla_i^\prime\nabla_j^\prime u$ with $\Delta_g u$. Let 
$B_{ij}{}^k\in\calE_{(ij)}{}^k$ be the difference tensor defined by $\nabla^g-\nabla^\prime$. 
The flatness of $[\nabla]$ implies $\nabla^\prime_{[i}g_{j]k}=-(1/2(n-1))\nabla^\prime_l C_{ij}{}^l{}_k=0$, so
it holds that $B_{ijk}=1/2\nabla^\prime_ig_{jk}\in\calE_{(ijk)}$, where the index is lowered by $g_{ij}$. 
Hence we have 
$$
g^{ij}B_{ij}{}^k=g^{kl}B_{li}{}^i=g^{kl}vol_g^{-1}\nabla^\prime_l vol_g.
$$
The Monge--Amp\`ere equation gives
$$
vol_g=(-2\br)^{-(n/2+1)}\Bigl(1-\frac{1}{2}\mathcal{O}\br^{n/2+1}+O(\br^{n/2+3})\Bigr).
$$
We take a projective scale $\tau$ on $N$ and set $\r=\tau^{-2}\br$. Then the associated connection
$\nabla$ satisfies $\nabla=\nabla^\prime+d\r/2\r$, so we have
\begin{align*}
g^{kl}vol_g^{-1}\nabla_l^\prime vol_g&=g^{kl}vol_g^{-1}\nabla_l vol_g
+\Bigl(\frac{n}{2}+1\Bigr)g^{kl}(\r^{-1}\nabla_l\r) vol_g \\
&=g^{kl}\nabla_l\log \Bigl(1-\frac{1}{2}\mathcal{O}\br^{n/2+1}+O(\br^{n/2+3})\Bigr).
\end{align*}
By the definition of the Blaschke metric, we have $g^{kl}=O(\br)$. We also have $g^{kl}\nabla_l\r=O(\br^2)$ as follows:  In the projective scale $\tau^\prime$, the top slot of $\wt g^{KL}D_L\r$ equals 
$(-2\br)^{-1}g^{kl}\nabla_l \r$. On the other hand, $\wt g^{KL}D_L\r=\tau^{-2}T^K+O(\br)$. Thus 
we have $g^{kl}\nabla_l\r=O(\br^2)$. It follows from these estimates that 
\begin{equation}\label{trB}
\begin{aligned}
g^{ij}B_{ij}{}^k&=-\frac{1}{2}g^{kl}\nabla_l(\mathcal{O}\br^{n/2+1})+O(\br^{n/2+4}) \\
&=O(\br^{n/2+2}).
\end{aligned}
\end{equation}
Thus, by the assumption that $du=O(\br^{-1})$, we obtain
$-g^{ij}\nabla_i^\prime\nabla_j^\prime u=-g^{ij}\nabla_i^g\nabla_j^g u-g^{ij}B_{ij}{}^k\nabla_k u
=\Delta_g u+O(\br^{n/2+1})$, which yields \eqref{Lap-relate}. 

(ii) As in the proof of (i), we compute as
\begin{equation}\label{Delta-f}
\begin{aligned}
\wt\Delta\wt f&=-\wt g^{IJ}D_ID_J\wt f+\wt g^{IJ}\wt\Gamma_{IJ}{}^KD_K\wt f \\
&=(n+2)(-\br)^{-1}\wt f-\frac{1}{2}g^{ij}\nabla^\prime_i\nabla^\prime_j\bigl((-\br)^{-1}\wt f\,\bigr)
-\Bigl(\frac{n}{2}+1\Bigr)\mathcal{O}\wt f\br^{n/2} \\
&\quad -\frac{1}{2}\wt g^{IJ}D_I\mathcal{O}D_J\wt f \br^{n/2+1}+O(\br^{n/2+2}).
\end{aligned}
\end{equation}
By the first equality in \eqref{trB}, we have
\begin{align*}
-g^{ij}\nabla^\prime_i\nabla^\prime_j\bigl((-\br)^{-1}\wt f\,\bigr)=\Delta_g\bigl((-\br)^{-1}\wt f\,\bigr)
-\frac{1}{2}g^{ij}\nabla^\prime_i\mathcal{O}\nabla^\prime_j\wt f \br^{n/2}+O(\br^{n/2+2}).
\end{align*}
Also, calculating in the projective scale $\tau^\prime$, we have 
$$
-\wt g^{IJ}D_I\mathcal{O}D_J\wt f \br^{n/2+1}=(n+2)\mathcal{O}\wt f\br^{n/2}
+\frac{1}{2}g^{ij}\nabla^\prime_i\mathcal{O}\nabla^\prime_j\wt f \br^{n/2}.
$$
Substituting these equations to \eqref{Delta-f}, we obtain \eqref{Lap-relate2}.
\end{proof}

\begin{prop}
{\rm (i)} The GJMS operator $P_m$ is self-adjoint for $1\le m\le n/2$ when $n$ is even and for all $m$ when $n$ is odd. \\
{\rm (ii)} Let $n$ be even and let $\br$ be a strict Fefferman defining density. Then $P_{n/2+2}$ is self-adjoint.
\end{prop}
\begin{proof}
{\rm (i)} We fix a projective scale $\tau\in\calE(1)$ and take an adapted frame $\{e_\infty=\xi, e_\a\}$ with the dual
$\{d\r, \th^\a\}$.
We set 
\begin{align*}
\wt h_{\a\b}&=h_{\a\b}+2\r P_{\a\b}, \quad \wt\th^\a=\th^\a+2\r\wt h^{\a\g}P_{\g\infty}d\r, \\
\wt r&=r+2\r P_{\infty\infty}-4\r^2\wt h^{\g\mu}P_{\g\infty}P_{\mu\infty},
\end{align*}
where $\wt h^{\a\b}$ is the inverse of $\wt h_{\a\b}$.
Then the Blaschke metric is represented as
$$
g=\frac{\wt h_{\a\b}}{-2\r}\,\wt\th^\a\cdot\wt\th^\b+\frac{1-2\wt r\r}{4\r^2}\,d\r^2.
$$
Thus the outward unit normal filed on each level set of $\r$ is given by
\begin{equation}\label{normal}
\begin{aligned}
\nu&=\frac{-2\r}{(1-2\wt r\r)^{1/2}}(\xi-2\r\wt h^{\a\g}P_{\g\infty}e_\a) \\
&\equiv-2\r(1+O(\r))\xi \quad {\rm mod}\ {\rm Ker}\,d\r.
\end{aligned}
\end{equation}
We identify a neighborhood of $M$ in $\overline\Omega$ with the product $M\times(-\e_0, 0]$ via the flow generated by $\xi$ so that the second component is given by $\r$. Since
\begin{equation}\label{vol-g}
vol_g=\frac{1+O(\r)}{(-2\r)^{n/2+1}}d\r\wedge vol_h,
\end{equation}
we have 
\begin{equation}\label{nu-vol}
(\nu\lrcorner\, vol_g)|_{TM_\e}=\frac{1+O(\e)}{(2\e)^{n/2}}vol_h,
\end{equation}  
where $M_\e=\{\r=-\e\}$. 

Let $f_1, f_2\in\calE[m-n/2]$. We take $A_j\in\calE(m-n/2)$ and $B_j\in\calE(-m-n/2)$ for $j=1, 2$ as in Proposition \ref{harm-ext} (ii), and set 
$$
u_j=(-\br)^{1/2(n/2-m)}(A_j+B_j\br^m\log|\r|).
$$
Since $u_j$ satisfies the assumption in Lemma \ref{Lap-relate-lem} (i), we have 
\begin{align*}
\bigl(\Delta_g+(m-n/2)(m+n/2)\bigr)u_j&=
2(-\br)^{1-1/2(m-n/2)}\wt\Delta(A_j+B_j\br^m\log|\r|) \\
&\quad +O(\r^{n/2+1}) \\
&=O(\r^{n/2+1})
\end{align*}
for even $n$. When $n$ is odd, the last term is replaced by $O(\br^\infty)$. We shall compute 
\begin{equation}\label{lp-int1}
{\rm lp}\int_{\{\r<-\e\}}\bigl(\langle du_1,du_2\rangle_g+(m-n/2)(m+n/2)u_1u_2\bigr)vol_g,
\end{equation}
where `${\rm lp}$' stands for the coefficient of $\log (1/\e)$ in the expansion of the integral.
By Green's formula, this is equal to
$$
{\rm lp}\int_{\{\r<-\e\}} u_1\bigl(\Delta_g+(m-n/2)(m+n/2)\bigr)u_2 vol_g
+{\rm lp}\int_{M_\e}(u_1\cdot\nu u_2) \nu\lrcorner \,vol_g.
$$ 
The first term equals 0 since the integrand is $O(1)$ for $1\le m\le n/2$ when $n$ is even and for all $m$ when $n$ is odd. It follows from \eqref{normal} and 
\eqref{nu-vol} that the second term is given by
$$
 2^{-n/2}(-1)^mc^{\prime-1}_m\int_M \bigl((n/2+m)f_1 P_m f_2+(n/2-m)f_2P_m f_1 \bigr).
$$
This integral must be symmetric in $f_1$ and $f_2$ since \eqref{lp-int1} is symmetric in $u_1$ and $u_2$. 
Thus we have 
$$
\int_M (f_1 P_m f_2-f_2 P_m f_1)=0,
$$ 
so $P_m$ is self-adjoint. \\
{\rm (ii)} Let $f_1, f_2\in\calE[2]$ and take 
$\wt f_1, \wt f_2\in\calE(2)$ as in Proposition \ref{harm-ext} (i), and set 
$u_j=(-\br)^{-1}\wt f_j$ ($j=1, 2$). By Lemma \ref{Lap-relate-lem} (ii), we have 
\begin{align*}
\bigl(\Delta_g+2(n+2)\bigr)u_j&=2\wt\Delta\bigl((-\br)u_j\bigr)+O(\br^{n/2+2}) \\
&=2\psi_j\br^{n/2+1}+O(\br^{n/2+2}).
\end{align*}
Hence, 
\begin{align*}
\langle du_1, du_2\rangle+2(n+2)u_1u_2
&=u_1\bigl(\Delta_g u_2+2(n+2)u_2\bigr)+ ({\rm div}) \\
&=-2\wt f_1\psi_2\br^{n/2}+O(\br^{n/2+1})+ ({\rm div}).
\end{align*}
As in the proof of (i), we see the log term of the integration of this function. 
Since the divergence term does not contribute to it, we have 
\begin{align*}
&\quad {\rm lp} \int_{\{\r<-\e\}}\bigl(\langle du_1, du_2\rangle+2(n+2)u_1u_2\bigr)vol_g \\
&={\rm lp}\int_{\{-\e_0<\r<-\e\}}\bigl(-2\wt f_1\psi_2\br^{n/2}+O(\br^{n/2+1})\bigr)\frac{1+O(\r)}{(-2\r)^{n/2+1}}d\r\wedge vol_h \\
&=c\int_M f_1P_{n/2+2}f_2,
\end{align*}
with a nonzero constant $c$. It follows from the symmetry in $u_1$ and $u_2$ that $P_{n/2+2}$ is self-adjoint. 
\end{proof}
\subsection{Conformal Codazzi $Q$-curvature}
Let $n$ be even. We take a conformal scale $\sigma\in\calE[1]$ on $M$ and extend it arbitrarily to a projective scale $\tau\in\calE(1)$ on $N$. We define the {\it (conformal Codazzi) $Q$-curvature} by  
$$
Q=-\wt\Delta^{n/2}\log t \,\big |_M\in\calE[-n],
$$
where $t$ is the fiber coordinate of the $\R_+$-bundle $\wt N$ associated with the trivialization by  $\tau^{-1}$. Note that the definition is independent of the choice of $\tau$; If we take a different extension, 
then $\log t$ changes by a function $f\in\calE(0)$ with $f=O(\br)$, so $Q$ is invariant.

Let $\wh\sigma=e^{-\U}\sigma$ be another conformal scale. We extend $\wh\sigma$ to $\wh\tau=e^{-\wt\U}\tau$ with an extension $\wt\U$ of $\U$. Then the $Q$-curvature changes as
\begin{align*}
\wh Q&=-\wt\Delta^{n/2}(\log t-\wt\U)\,\big|_M \\
&=Q+P_{n/2}\U.
\end{align*}
Since $P_{n/2}$ is self-adjoint and satisfies $P_{n/2}1=0$, the total $Q$-curvature 
$$
\overline{Q}=\int_M Q
$$
is a conformal invariant.

As in the usual conformal case (\cite{FH}), the $Q$-curvature is related to the harmonic extension of 
$\log t$: 
\begin{lem}
Let $\tau\in\calE(1)$ be a projective scale and set $\r=\tau^{-2}\br$, where $\br$ is a Fefferman defining 
density. Then there exist $A\in\calE(0)$ and $B\in\calE(-n)$ such that  $A|_M=0$ and 
\begin{equation}\label{Delta-logt}
\wt\Delta(\log t+A+B\br^{n/2}\log|\r|)=O(\br^{\infty}).
\end{equation}
Such $A$ and $B$ are unique modulo $O(\br^{n/2})$ and $O(\br^{\infty})$ respectively and it holds that $2^{n/2-1}n\{(n/2-1)!\}^2B|_M=Q$.
\end{lem}

With this lemma, we show the following proposition:

\begin{prop}
There exists a nonzero universal constant $a_n$ such that 
$$
L=a_n\overline{Q},
$$
where $L$ is the coefficient of the logarithmic term in the volume expansion of the Blaschke metric $g$.
\end{prop}
\begin{proof}
We fix a projective scale and take $A$ and $B$ as in the previous lemma. It follows from Lemma \ref{Lap-relate-lem} (i) that  
\begin{align*}
\Delta_g\log|\r|&=-2\br\wt\Delta\log|\r|+O(\br^{n/2+1}) \\
&=-2\br\wt\Delta(\log |\br|-2\log t)+O(\br^{n/2+1}) \\
&=2n+4\br\wt\Delta\log t+O(\br^{n/2+1}).
\end{align*}
Then, multiplying both sides of \eqref{Delta-logt} by $-2\br$ and using Lemma \ref{Lap-relate-lem} (i) again, we have
$$
\Delta_g(1/2\log|\r|-A-B\br^{n/2}\log|\r|)=n+O(\br^{n/2+1}).
$$
Hence by \eqref{normal}, \eqref{vol-g} and \eqref{nu-vol}, we obtain 
\begin{align*}
L&={\rm lp}\int_{\{\r<-\e^2/2\}} vol_g \\
&={\rm lp}\,\frac{2}{n}\int_{M_{\e}}\nu\cdot\Bigl(\frac{1}{2}\log|\r|-A-B\br^{n/2}\log|\r|\Bigr)\nu\lrcorner \,vol_g \\
&=a_n \overline{Q}
\end{align*}
with a nonzero constant $a_n$.
\end{proof}

\section{Variation formulas}

\subsection{The first variation formula}
Let $N$ be a locally flat projective manifold of odd dimension $n+1$, and let $\Omega$ be a relatively compact domain in $N$ with strictly convex boundary $M$. We consider a smooth family of strictly convex domains $\{\Omega_t\}$ with $\Omega_0=\Omega$ and compute $(d/dt)|_{t=0}\overline{Q_t}$, where $\overline{Q_t}$ is the total conformal Codazzi $Q$-curvature of 
$\partial\Omega_t$. Since $\overline{Q_t}$ agrees with the log term $L_t$ in the volume expansion of the Blaschke metric, we will compute $\dot L=(d/dt)|_{t=0}L_t$ instead.
We take a family of Fefferman defining densities $\br_t$ of $\Omega_t$ which is smooth in $t$ and denote the associated obstruction density and the Blaschke metric by $\mathcal{O}_t$ and $g^t_{ij}$ respectively. We will omit the subscript or superscript $t$ when $t=0$.

The logarithmic term in the volume expansion can be described in terms of Riesz renormalization 
(see, e.g., \cite{Al} for the details):  
We fix a projective scale $\tau\in\calE(1)$ and set   
$$
\zeta_t(z)=\int_{\Omega_t} (-\r_t)^z vol_{g^t},
$$
where $\r_t=\tau^{-2}\br_t$. Then $\zeta_t(z)$ is holomorphic in $z$ if ${\rm Re} z$ is sufficiently large 
and extends to a meromorphic function on $\mathbb{C}$. Moreover, it has at most simple pole at $z=0$ and satisfies
$$
\underset{z=0}{\rm Res}\, \zeta_t(z)={\rm lp}\int_{\{\r_t<-\e\}} vol_{g^t}=\frac{1}{2}L_t.
$$
Hence we have 
\begin{align*}
\frac{1}{2}\dot L&=\underset{z=0}{\rm Res}\,\dot\zeta_0(z) \\
&=\underset{z=0}{\rm Res}\,\int_{\Omega}\frac{d}{dt}\Big|_{t=0}(-\r_t)^z vol_{g^t} \\
&\quad +\underset{z=0}{\rm Res}\,\lim_{t\to 0}\frac{1}{t}\Bigl(\int_{\Omega_t}(-\r_t)^z vol_{g^t}
-\int_\Omega(-\r_t)^z vol_{g^t}\Bigr).
\end{align*}
The limit in the second term in the last expression equals 0 when ${\rm Re} z$ is large enough since $(-\r_t)^z (\chi_{\Omega_t}-\chi_{\Omega})vol_{g^t}$ is uniformly convergent to 0 in a neighborhood 
of $\partial\Omega$, where $\chi_{\Omega_t}$ and $\chi_{\Omega}$ are the characteristic functions. The first term is computed as 
\begin{align*}
\underset{z=0}{\rm Res}\,\int_{\Omega}\frac{d}{dt}\Big|_{t=0}(-\r_t)^z vol_{g^t}&=
\underset{z=0}{\rm Res}\,z\int_{\Omega}(-\r)^z \frac{\dot\r}{\r}\,vol_{g}
+\underset{z=0}{\rm Res}\,\int_{\Omega}(-\r)^z \frac{d}{dt}\Big|_{t=0}vol_{g^t} \\
&=\underset{z=0}{\rm Res}\,\int_{\Omega}(-\r)^z \frac{d}{dt}\Big|_{t=0}vol_{g^t}.
\end{align*}
Here we have used the fact that the integral 
$$
\int_{\Omega}(-\r)^z \frac{\dot\r}{\r}\,vol_{g}
$$
has at most simple pole at $z=0$. Thus we obtain the following expression for $\dot L$:
$$
\dot L=2\,{\rm lp}\int_{\{\r<-\e\}} \frac{d}{dt}\Big|_{t=0}vol_{g^t}.
$$
The approximate Monge--Amp\`ere equation for $\br_t$ gives
$$
\det g^t_{ij}=(-2\br_t)^{-(n+2)}(1-\mathcal{O}_t\br_t^{n/2+1}),
$$ 
so we have
\begin{align*}
 \frac{d}{dt}\Big|_{t=0}vol_{g^t}&= \frac{d}{dt}\Big|_{t=0}(\det g^t_{ij})^{1/2} \\
&=-\Bigl(\frac{n}{2}+1\Bigr)\Bigl(\,\frac{\dot\br}{\br}+\frac{1}{2}\mathcal{O}\dot\br\br^{n/2}+O(\br^{n/2+1})\Bigr)vol_g.
\end{align*}
We shall compute $\dot\br/\br$ by expressing $\wt g^{IJ}D_ID_J\dot\br$ in two ways. First we let 
$\nabla^\prime\in[\nabla]$ be the representative connection with respect to the singular projective scale 
$\tau^\prime=(-\br)^{1/2}$, as in the proof of Lemma \ref{Lap-relate-lem} (i). Then near the boundary 
it holds that 
\begin{equation*}
\wt g^{IJ}D_ID_J\dot\br=(n+2)\frac{\dot\br}{\br}-\frac{1}{2}g^{ij}\nabla^\prime_i\nabla^\prime_j
\Bigl(\,\frac{\dot\br}{\br}\,\Bigr).
\end{equation*}
Using the first equality of \eqref{trB} and the estimates $g^{kl}\nabla_l\br=O(\br^2)$, $g^{kl}\nabla_k\br\nabla_l\br=4\br^{2}(1+O(\br))$, we have
\begin{align*}
g^{ij}\nabla^\prime_i\nabla^\prime_j\Bigl(\,\frac{\dot\br}{\br}\,\Bigr)
+\Delta_g\Bigl(\,\frac{\dot\br}{\br}\,\Bigr)&=g^{ij}B_{ij}{}^k\nabla_k\Bigl(\,\frac{\dot\br}{\br}\,\Bigr) \\
&=-\frac{1}{2}\Bigl(\frac{n}{2}+1\Bigr)\mathcal{O}\br^{n/2}g^{kl}\nabla_l\br\Bigl(\frac{\nabla_k\dot\br}{\br}-\frac{\dot\br\nabla_k\br}{\br^2}\Bigr)+O(\br^{n/2+1}) \\
&=(n+2)\mathcal{O}\dot\br\br^{n/2}+O(\br^{n/2+1}).
\end{align*}
Hence 
\begin{equation}\label{Lap-dot1}
\wt g^{IJ}D_ID_J\dot\br=(n+2)\frac{\dot\br}{\br}-\Bigl(\frac{n}{2}+1\Bigr)\mathcal{O}\dot\br\br^{n/2}+\frac{1}{2}\Delta_g\Bigl(\,\frac{\dot\br}{\br}\,\Bigr)+O(\br^{n/2+1}).
\end{equation}
On the other hand, by differentiating the Monge--Amp\`ere equation $\det(D_ID_J\br_t)=-1+\mathcal{O}_t\br_t^{n/2+1}$, we have
\begin{equation}\label{Lap-dot2}
\begin{aligned}
\wt g^{IJ}D_ID_J\dot\br&=\bigl(-1+O(\br^{n/2+1})\bigr)^{-1}\Bigl(\Bigl(\frac{n}{2}+1\Bigr)
\mathcal{O}\dot\br\br^{n/2}+O(\br^{n/2+1})\Bigr) \\
&=-\Bigl(\frac{n}{2}+1\Bigr)\mathcal{O}\dot\br\br^{n/2}+O(\br^{n/2+1}).
\end{aligned}
\end{equation}
Comparing \eqref{Lap-dot1} and \eqref{Lap-dot2} gives
$$
\frac{\dot\br}{\br}=-\frac{1}{2(n+2)}\Delta_g
\Bigl(\,\frac{\dot\br}{\br}\,\Bigr)+O(\br^{n/2+1}),
$$
which implies that the integral of $\dot\br/\br$ does not yield a logarithmic term.
Consequently the variation of $L$ is computed as
\begin{align*}
\dot L&=-\Bigl(\frac{n}{2}+1\Bigr)\,{\rm lp}\int_{\{\r<-\e\}} \mathcal{O}\dot\br\br^{n/2} vol_g \\
&=-\Bigl(\frac{n}{2}+1\Bigr)\,{\rm lp}\int_{\{-\e_0<\r<-\e\}} \mathcal{O}\dot\br\br^{n/2}
\frac{1+O(\r)}{(-2\r)^{n/2+1}}d\r\wedge vol_h \\
&=k_n\int_M \mathcal{O}\dot\br,
\end{align*}
where $k_n=(n/2+1)(-2)^{-n/2-1}$.

Thus we have proved the following
\begin{thm}\label{first-var}
Let $\{\Omega_t\}$ be a smooth family of strictly convex domains in a locally flat projective manifold $N$ of 
odd dimension $n+1$. Let $\br_t$ be a Fefferman defining density of $\Omega_t$. Then the total conformal Codazzi $Q$-curvatures $\overline{Q_t}$ on $\partial\Omega_t$ satisfy 
$$
\frac{d}{dt}\Big|_{t=0} \overline{Q_t}=k_n\int_M \mathcal{O}\dot\br,
$$
where $\mathcal{O}$ is the obstruction density on $M=\partial\Omega_0$ and $k_n$ is a nonzero universal constant.
\end{thm} 

Consequently, the critical point of $\overline{Q}$ is given by hypersurfaces whose obstructions vanish. 
As an example, let us compute the obstruction density for a strictly convex surface $M$ in $\R^3$. 
We take the canonical affine scale $\tau=(dx^1\wedge dx^2\wedge dx^3)^{-1/4}$, for which   
$P_{ij}=0$. Then we differentiate \eqref{r} and use \eqref{xi-S} to obtain 
$2\underline{\mathcal{O}}|_M=-|S|^2-r|_M{\rm tr}S+4r^2|_M+({\rm div})$. Since $\tau$ is harmonic, 
$2r|_M={\rm tr}S=\scal-|A|^2$ by \eqref{r-trS}. Noting that 
${\rm tf}\,\Ric_{\a\b}^h={\rm tf}A_{\a\mu\nu}A_{\b}{}^{\mu\nu}=0$ in dimension 2, we also have 
${\rm tf}S_{\a\b}=(\d A)_{\a\b}$ from \eqref{Pab-M}. Therefore the obstruction is given by 
$$
\underline{\mathcal{O}}|_M=-\frac{1}{2}|\d A|^2+({\rm div}).
$$
Integrating this over $M$, we see that $\mathcal{O}\equiv0$ if and only if $\d A\equiv0$, or equivalently 
${\rm tf}S\equiv0$. A hypersurface with ${\rm tf}S\equiv0$ in affine scales is known as an {\it affine sphere}. It is a classical theorem of Blaschke ($n=2$) and Deicke ($n\ge3$) that a strictly convex affine 
sphere in $\R^{n+1}$ is projectively equivalent to the sphere (see \cite{NS} for a proof). Thus we obtain the 
following corollary:
\begin{cor}\label{cor-first-var}
A strictly convex surface in $\R^3$ is a critical point of $\overline{Q}$ if and only if it is projectively equivalent to $S^2$.
\end{cor}

\subsection{The second variation formula}
We will derive a second variation formula of $\overline{Q}$ for the deformation of $\Omega$ parametrized by a 
density on the boundary $M$. Let $\br$ be a strict Fefferman defining density of $\Omega$. 
We fix a conformal density $f\in\calE[2]$ and extend it to a projective density $\wt f\in\calE(2)$ with $\wt\Delta\wt f=O(\br)$, where 
$\wt\Delta$ is the Laplacian of the ambient metric $\wt g_{IJ}=D_ID_J\br$. We set
$$
X^I=\wt g^{IJ}D_J\wt f\in\calE^I(1).
$$
Recall that the tractor field $X^I$ can be identified with a vector filed $X$ on $\wt N$, the $\R_+$-bundle of positive elements in $\calE(-1)$.  Since $X$ is homogeneous of 
degree 0, it projects to a vector field $\underline{X}$ on $N$. Let ${\rm Fl}_t$ and $\underline{\rm Fl}_t$
be the flows generated by $X$ and $\underline{X}$ respectively. We consider the family of strictly convex domains $\{\Omega_t\}$ defined by $\Omega_t=\underline{\rm Fl}_t(\Omega)$.
\begin{lem}\label{def-density}
There exist Fefferman defining densities $\br_t$ of $\Omega_t$ which satisfy $\br_0=\br$ and 
$\dot\br:=(d/dt)|_{t=0}\,\br_t=-2\wt f+O(\br)$.
\end{lem}

\begin{proof}
We put $\br^\prime_t={\rm Fl}^\ast_{-t}\br$. Then it holds that $\br^\prime_0=\br$ and 
$\dot\br^\prime=-X^ID_I\br=-2\wt f$. Since $\wt\Delta\wt f=O(\br)$, we have
$$
\frac{d}{dt}\Big|_{t=0}\calJ[\br^\prime_t]=\calJ[\br]\,\wt g^{IJ}D_ID_J\dot\br^\prime=O(\br),
$$
which implies $\calJ[\br^\prime_t]\big|_M=-1+O(t^2)$. Starting with $\br^\prime_t$, we construct a Fefferman defining density $\br^{\prime\prime}_t$ by the algorithm described in the proof of 
\cite[Proposition 3.3]{M}. Then $\br^{\prime\prime}_t$ is written as
$$
\br^{\prime\prime}_t=(-\calJ[\br^\prime_t])^{-(1/n+2)}\br^\prime_t+\phi_t(\br^\prime_t)^2
$$ 
with some density $\phi_t$, so it holds that $\dot\br^{\prime\prime}=-2\wt f+O(\br)$. 
Since $\br$ and $\br^{\prime\prime}_0$ are both Fefferman defining densities, we can write as 
$\br^{\prime\prime}_0=\br+\psi\br^{n/2+2}$. Thus, setting  $\br_t=\br^{\prime\prime}_t
-{\rm Fl}_{-t}^\ast(\psi\br^{n/2+2})$, we obtain a family of Fefferman defining densities with $\br_0=\br$
 and $\dot\br=-2\wt f+O(\br)$.
\end{proof}

First we present a variation formula of the obstruction density.

\begin{prop}\label{prop-O-var}
Let $\br_t$ be the Fefferman defining density of $\Omega_t$ given by the previous lemma. Then the obstruction densities $\mathcal{O}_t$ of $\br_t$ satisfy
\begin{equation}\label{O-var}
\frac{d}{dt}\Big|_{t=0} {\rm Fl}_t^\ast\mathcal{O}_t|_M=b_n P_{n/2+2}f,
\end{equation}
where $b_n$ is a nonzero universal constant and $P_{n/2+2}$ is the GJMS operator.
\end{prop}

\begin{proof}
Differentiating both sides of $\calJ[\br_t]=-1+\mathcal{O}_t\br_t^{n/2+1}$ at $t=0$, we have 
$$
\wt g^{IJ}D_ID_J\dot\br=-\Bigl(\frac{n}{2}+1\Bigr)\mathcal{O}\dot\br\br^{n/2}
-\dot{\mathcal{O}}\br^{n/2+1}+O(\br^{n/2+2}).
$$
It follows from this equation and \eqref{trGamma} that 
\begin{align*}
\wt\Delta\dot\br&=-\wt g^{IJ}D_ID_J\dot\br+\wt g^{IJ}\wt\Gamma_{IJ}{}^KD_K\dot\br \\
&=\Bigl(\dot{\mathcal{O}}-\frac{1}{2}\,\wt g^{IJ}D_I\mathcal{O}D_J\dot\br\Big)\br^{n/2+1}+O(\br^{n/2+2}).
\end{align*}
Then, since $\dot\br$ and $\wt f$ satisfy $\dot\br+2\wt f=O(\br)$ and $\wt\Delta(\dot\br+2\wt f\,)=O(\br)$, it holds that $\dot\br=-2\wt f+O(\br^2)$. Thus we have
\begin{equation*}
\wt\Delta\dot\br=(\dot{\mathcal{O}}+X^ID_I\mathcal{O})\br^{n/2+1}+O(\br^{n/2+2}).
\end{equation*}
Applying $\wt\Delta^{n/2+1}$ to both sides yields
$$
b_n P_{n/2+2}f=(\dot{\mathcal{O}}+X^ID_I\mathcal{O})|_M
=\frac{d}{dt}\Big|_{t=0}{\rm Fl}_t^\ast\mathcal{O}_t|_M,
$$
with a nonzero constant $b_n$. 
\end{proof}

Next, we prepare a lemma on the derivatives of a density $\psi\in\calE(-n-2)$. 
Since $\psi$ is identified with a homogeneous function on $\wt N$, one can consider the derivative 
$X\psi\in\calE(-n-2)$. On the other hand $\psi$ can also be regarded as a form on $N$ via the 
isomorphism $\calE(-n-2)\cong\wedge^{n+1}T^\ast N$, so one can consider the Lie derivative 
$\mathcal{L}_{\underline{X}}\psi$. The following lemma asserts that the two derivatives coincide at $M$.

\begin{lem}\label{density-derivative}
For a projective density $\psi\in\calE(-n-2)$, we have 
$$
\mathcal{L}_{\underline{X}}\psi=X\psi
$$
at $M$ under the identification $\calE(-n-2)\cong\wedge^{n+1} T^\ast N$.
\end{lem}
\begin{proof}
We take a projective scale $\tau\in\calE(1)$ and let $t\in\R_+$ be the fiber coordinate of $\wt N$ 
associated with the trivialization by $\tau^{-1}$. It can be seen that the determinant of a tractor in $\calE_{IJ}$ defined in \cite{M} is identical to the one with respect to the volume form 
$$
\Psi=t^{n+1}dt\wedge\tau^{-(n+2)}
$$
on $\wt N$. Here we regard $\tau^{-(n+2)}$ as a form on $\wt N$ by pulling it back by the 
projection map. By the Monge--Amp\`ere equation, the volume form of $\wt g_{IJ}$ satisfies
$$
vol_{\wt g}=\bigl(1+O(\br^{n/2+1})\bigr)\Psi.
$$
Then it follows from $\mathcal{L}_X vol_{\wt g}=-\wt\Delta\wt f \,vol_{\wt g}=O(\br)$ that 
$\mathcal{L}_X\Psi=O(\br)$. Since $X$ is homogeneous of degree 0, there is a function $a$ on $N$ such that $X=at(\partial/\partial t)+\underline{X}$. Using $Xt=at$ and $\mathcal{L}_X dt=d(at)$, we have
$$
\mathcal{L}_X\Psi=t^{n+1}dt\wedge\bigl(\mathcal{L}_{\underline{X}}\tau^{-(n+2)}+(n+2)a\tau^{-(n+2)}\bigr),
$$
and thus
\begin{equation*}
\mathcal{L}_{\underline{X}}\tau^{-(n+2)}|_M=-(n+2)a\tau^{-(n+2)}|_M.
\end{equation*}
If we set $\underline{\psi}=\tau^{n+2}\psi$, the density $\psi$ is identified with the function $t^{-(n+2)}\underline{\psi}$. Then it holds at $M$ that 
$$
\mathcal{L}_{\underline{X}}\psi=(\underline{X}\underline{\psi})\,\tau^{-(n+2)}+\underline{\psi}
\mathcal{L}_{\underline{X}}\tau^{-(n+2)} 
=\bigl(\underline{X}\underline{\psi}-(n+2)a\underline{\psi}\bigr)\tau^{-(n+2)},
$$
and that 
$$
X\psi=t^{-(n+2)}\bigl(\underline{X}\underline{\psi}-(n+2)\,t^{-1}(Xt)\underline{\psi}\bigr) 
=t^{-(n+2)}\bigl(\underline{X}\underline{\psi}-(n+2)a\underline{\psi}\bigr).
$$
Thus $\mathcal{L}_{\underline{X}}\psi|_M$ and $X\psi|_M$ coincide under the identification 
$\calE(-n-2)\cong\wedge^{n+1}T^\ast N$. 
\end{proof}

Now we are ready to show the following

\begin{thm}\label{L-second-var}
Let $\br$ be a strict Fefferman defining density of a strictly convex domain $\Omega$ in a locally flat projective manifold $N$ of odd dimension $n+1$. If $\{\Omega_t\}$ is the deformation of 
$\Omega$ defined from $f\in\calE[2]$, the total conformal Codazzi $Q$-curvatures $\overline{Q_t}$ 
on $\partial\Omega_t$ satisfy
\begin{equation}\label{second-var}
\frac{d^2}{dt^2}\Big|_{t=0}\overline{Q_t}=k_n^\prime\int_M f P_{n/2+2}f,
\end{equation}
where $k_n^\prime$ is a nonzero universal constant and $P_{n/2+2}$ is the GJMS operator.
\end{thm}
\begin{proof}
We fix a projective scale $\tau\in\calE(1)$ and take Fefferman defining densities $\br_t=\tau^2\r_t$ 
with the obstruction densities $\mathcal{O}_t=\tau^{-(n+2)}\underline{\mathcal{O}_t}$, as in Lemma \ref{def-density}. Recall from the proof of Proposition \ref{prop-O-var} that the derivative $\dot\br$ satisfies $\dot\br=-2\wt f+O(\br^2)$. 

By Theorem \ref{first-var}, we have
$$
\frac{d}{dt}\overline{Q_t}=\int_{\partial\Omega_t} \mathcal{O}_t\dot\br_t
=\int_M \underline{\rm Fl}_t^\ast(\underline{\mathcal{O}_t}\dot\r_t\, vol_{h^t}),
$$
where $h^t$ is the affine metric on $\partial\Omega_t$ with respect to $\tau$. Thus the second derivative of $L_t$ is decomposed as $\ddot L=\rm (I)+(II)$ with
$$
{\rm(I)}=\int_M \frac{d}{dt}\Big|_{t=0}\underline{{\rm Fl}}_t^\ast(\underline{\mathcal{O}_t}\dot\r_t)\,vol_h,
\quad {\rm (II)}=\int_M \underline{\mathcal{O}}\dot\r\frac{d}{dt}\Big|_{t=0}\underline{{\rm Fl}}_t^\ast vol_{h^t}.
$$

First we compute (I). Using Proposition \ref{prop-O-var}, we have
\begin{align*}
\frac{d}{dt}\Big|_{t=0}\underline{{\rm Fl}}_t^\ast(\underline{\mathcal{O}_t}\dot\r_t)|_M&=
\frac{d}{dt}\Big|_{t=0}{\rm Fl}_t^\ast(\tau^n\mathcal{O}_t\dot\br_t)|_M \\
&=\mathcal{O}\dot\br(X\tau^n)|_M-2b_n\tau^n fP_{n/2+2}f+
\tau^n\mathcal{O}\frac{d}{dt}\Big|_{t=0}{\rm Fl}_t^\ast\dot\br_t|_M.
\end{align*}
Differentiating ${\rm Fl}_t^\ast\br_t=O(\br)$ gives ${\rm Fl}_t^\ast\dot\br_t+
{\rm Fl}_t^\ast(X^ID_I\br_t)=O(\br)$, so it holds that 
\begin{align*}
\frac{d}{dt}\Big|_{t=0}{\rm Fl}_t^\ast\dot\br_t&=-X^ID_I\dot\br-X^JD_J(X^ID_I\br)+O(\br) \\
&=O(\br).
\end{align*}
Thus we obtain 
\begin{equation}\label{I}
{\rm (I)}=k_n^\prime\int_M f P_{n/2+2}f+\int_M \mathcal{O}\dot\br(X\tau^n)
\end{equation}
with a nonzero constant $k_n^\prime$.

Next we compute (II). For each $t$, we take a vector field $\xi_t$ on $N$ determined by the conditions
$\xi_t\r_t=1$ and $\xi_t^i Z^j\nabla_i\nabla_j\r_t=0$ for $Z\in{\rm Ker}\,d\r_t$, where $\nabla\in[\nabla]$ is the representative connection associated with $\tau$. Since $\br_t$ solves the Monge--Amp\`ere equation at $\partial\Omega_t$, the restriction $\xi_t|_{\partial\Omega_t}$ is affine normal. Thus, using Lemma 
\ref{density-derivative}, we compute as 
\begin{align*}
\frac{d}{dt}\Big|_{t=0}\underline{{\rm Fl}}_t^\ast vol_{h^t}&=
\frac{d}{dt}\Big|_{t=0}\underline{{\rm Fl}}_t^\ast (\xi_t\lrcorner\,\tau^{-(n+2)}) \\
&=(\dot\xi+\mathcal{L}_{\underline{X}}\xi)\lrcorner\,\tau^{-(n+2)}
+\xi\lrcorner\,\mathcal{L}_{\underline{X}}\tau^{-(n+2)} \\
&=d\r(\dot\xi+\mathcal{L}_{\underline{X}}\xi)vol_h-(n+2)\tau^{-1}X\tau \,vol_h.
\end{align*}
Differentiating the equation $\underline{{\rm Fl}}_t^\ast(\xi_t\r_t)=1$ and using $\dot\br=-2\wt f+O(\br^2)$ and $X\br=2\wt f$, we have 
\begin{align*}
d\r(\dot\xi+\mathcal{L}_{\underline{X}}\xi)&=-\xi(\dot\r+\underline{X}\r) \\
&=-\xi(\tau^{-2}\dot\br+X(\tau^{-2}\br)) \\
&=2\,\tau^{-1}X\tau+O(\r),
\end{align*}
which yields
\begin{equation}\label{II}
{\rm (II)}=-n\int_M\underline{\mathcal{O}}\dot\r\,\tau^{-1}X\tau\,vol_h
=-\int_M\mathcal{O}\dot\br(X\tau^n).
\end{equation}
Combining \eqref{I} and \eqref{II}, we obtain the formula \eqref{second-var}.
\end{proof}

As an example, we consider the case of the sphere $M=S^n$. Let $(\xi^0, \dots, \xi^{n+1})$ be the 
homogeneous coordinates of $\R\mathbb{P}^{n+1}$. The projective density bundle $\calE(-1)$ of 
 $\R\mathbb{P}^{n+1}$ is identified with the tautological line bundle, and the $D$-operator is given by $\partial/\partial\xi^I$.
The defining density $\br=(1/2)(-(\xi^0)^2+(\xi^1)^2+\cdots+(\xi^{n+1})^2)$ of $M$ is the exact solution 
to the Monge--Amp\`ere equation, so the ambient metric is given by 
$\wt g=-(d\xi^0)^2+(d\xi^1)^2+\cdots+(d\xi^{n+1})^2$. Thus $\wt g$ agrees with the Fefferman--Graham
 ambient metric in conformal geometry. In the standard scale on $S^n$, the GJMS operator is given 
explicitly by
$$
P_{n/2+2}=\prod_{j=1}^{n/2+2}\Bigl(\Delta_h+\Bigl(\frac{n}{2}+j-1\Bigr)\Bigl(\frac{n}{2}-j\Bigr)\Bigr),
$$ 
where $h$ is the standard metric on $S^n$; see \cite{FG2} for the derivation of this formula. 

Let $E_k$ be the eigenspace of $\Delta_h$ with eigenvalue $k(k+n-1)$ for $k\ge0$, so that we have the decomposition $L^2(S^n)=\sum_{k\ge0} E_k$.
It follows from the above formula that ${\rm Ker}\,P_{n/2+2}=E_0\oplus E_1\oplus E_2$ and $P_{n/2+2}$ 
is positive definite on $({\rm Ker}\,P_{n/2+2})^\perp$. 

The space $E_0\oplus E_1\oplus E_2$ has the following geometric interpretation: 
Since eigenfunctions of $\Delta_h$ are given by the restriction of homogeneous harmonic polynomials on 
$\R^{n+1}$, it holds that  
$$
E_0\oplus E_1\oplus E_2=\{F(\xi)=(a_{IJ}\xi^I\xi^J)|_{S^n}\ |\ a_{IJ}=a_{JI}\in\R\}.
$$ 
In fact, noting that we may assume that $\sum_{i=1}^{n+1}a_{ii}=0$ by adding a multiple of $\br$, we have 
$$
F=a_{00}+2a_{0i}x^i+a_{ij}x^ix^j,
$$
where $x^i=\xi^i/\xi^0$ are the affine coordinates, and each term in the right-hand side belongs to 
$E_0$, $E_1$, $E_2$ respectively. The gradient vector field of $F$ is given by
$$
X=-2a_{0J}\xi^J\frac{\partial}{\partial \xi^0}+2\sum_{i=1}^{n+1}a_{iJ}\xi^J\frac{\partial}{\partial\xi^i},
$$ 
so the flow generated by $X$ is of the form ${\rm Fl}_t(\xi)=(I+tA)\xi$ with some matrix $A$.  
Therefore, the flow $\underline{\rm Fl}_t$ on $\R\mathbb{P}^{n+1}$ is a projective linear transformation, which does not change the conformal Codazzi structure on $S^n$.


\begin{thebibliography}{ABC}
\bibitem[Al]{Al}P. Albin,
{\it Renormalizing curvature integrals on Poincar\'e-Einstein manifolds,}
Adv. Math. {\bf 221} (2009), 140--169.

\bibitem[BC]{BC}F. E. Burstall and D. M. J. Calderbank, 
{\it Submanifold geometry in generalized flag manifolds,}
in Winter School in Geometry and Physics (Srni, 2003), Rend. del Circ. mat. di Palermo {\bf 72} (2004), 13--41.


\bibitem[BEG]{BEG}T. N. Bailey, M. G. Eastwood, and A. R. Gover,
{\it Thomas' structure bundles for conformal, projective and related structures,}
 Rocky Mountain J. {\bf 24} (1994), 1191--1217.



\bibitem[C]{C}D. M. J. Calderbank,
{\it M\"obius structures and two-dimensional Einstein-Weyl geometry,}
J. Reine Angew. Math. {\bf 504} (1998), 37--53.


\bibitem[CY]{CY}S. Y. Cheng and S. T. Yau,
{\it On the existence of a complete K\"ahler-Einstein metric on noncompact complex manifolds and the regularity of Fefferman's equation,}
Comm. Pure Appl. Math. {\bf 123} (1980), 507--544.

\bibitem[Fe]{Fe} C. Fefferman,  
{\it Monge--Amp\`ere equations, the Bergman kernel, and geometry of pseudoconvex domains,} 
Ann. of Math. {\bf 103} (1976), 395--416. 

\bibitem[FG1]{FG1}C. Fefferman and C. R. Graham,
{\it Conformal invariants,}
in ``\'Elie Cartan er les Math\'ematiques d'Aujourd'hui,'' Ast\'erisque, hors s\'erie (1985), 95--116. 

\bibitem[FG2]{FG2}C. Fefferman and C. R. Graham,
{\it $Q$-curvature and Poincar\'{e} metric,} 
Math. Res. Lett. {\bf 10} (2003), 819--832.

\bibitem[FG3]{FG3}C. Fefferman and C. R. Graham,
{\it The ambient metric,}
Princeton University Press, 2011. 

\bibitem[FH]{FH}C. Fefferman and K. Hirachi,
{\it Ambient metric construction of $Q$-curvature in conformal and CR geometries,} 
Math. Res. Lett. {\bf 10} (2003), 819--832.

\bibitem[GG]{GG}A. R. Gover and C. R. Graham,
{\it CR invariant powers of the sub-Laplacian,}
J. Reine Angew. Math. {\bf 583} (2005), 1--27.

\bibitem[GH]{GH}A. R. Gover and K. Hirachi,
{\it Conformally invariant powers of the Laplacian---A complete non-existence theorem,}
Jour. Amer. Math. Soc. {\bf 17} (2004), 389--405.

\bibitem[G1]{G1}C. R. Graham,
{\it Conformally invariant powers of the Laplacian, II: Nonexistence,}
J. London Math. Soc. {\bf 46} (1992), 566--576.

\bibitem[G2]{G2}C. R. Graham, 
{\it Volume and area renormalizations for conformally compact Einstein metrics,}
Rend. Cirs. Mat. Palermo Ser. II {\bf 63} (2000), Suppl., 31--42.


\bibitem[GJMS]{GJMS}C. R. Graham, R. Jenne, L. J. Mason and G. A. J. Sparling,
{\it Conformally invariant powers of the Laplacian, I: Existence,}
J. London Math. Soc. {\bf 46} (1992), 557--565.

\bibitem[GZ]{GZ}C. R. Graham and M. Zworski,
{\it Scattering matrix in conformal geometry,}
Invent. Math. {\bf 152} (2003), 89--118.


\bibitem[HMM]{HMM}K. Hirachi, T. Marugame and Y. Matsumoto,
{\it Variations of total $Q$-prime curvature on CR manifolds,}
{\tt arXiv:1510.03221}

\bibitem[HPT]{HPT}P. Hislop, P. Perry, and A.-H. Tang, 
{\it CR-invariants and the scattering operator for complex manifolds with boundary,}
Analysis \& PDE {\bf 1} (2008), 197--227.

\bibitem[M]{M}T. Marugame,
{\it Volume renormalization for the Blaschke metric on strictly convex domains,}
{\tt arXiv:1601.06948}

\bibitem[NS]{NS}K. Nomizu and T. Sasaki,
{\it Affine Differential Geometry: Geometry of Affine Immersions,}
Cambridge University Press, 1994.

\end{thebibliography}
\end{document}